\documentclass[12pt]{amsart}
\usepackage{amsmath, amssymb}
\usepackage{mathtools, a4wide, mathrsfs, bm}
\usepackage{paralist}
\usepackage{xcolor}
\usepackage[T1]{fontenc}
\usepackage{geometry}

\usepackage[colorlinks=true,linkcolor=blue,citecolor=blue,pdfpagelabels=false]{hyperref}
\numberwithin{equation}{section}

\usepackage{paralist}

\usepackage[backend=biber, style=alphabetic, giveninits]{biblatex}
\hypersetup{
    colorlinks=true, 
    linktoc=all,     
    linkcolor=blue,  
}

\makeatletter

\renewcommand\subsubsection{\@secnumfont}{\bfseries}%
\renewcommand\subsubsection{\@startsection{subsubsection}{3}
  \z@{.5\linespacing\@plus.7\linespacing}{-.5em}%
  {\normalfont\bfseries}}
  
\makeatother

\mathtoolsset{showonlyrefs}

\theoremstyle{plain}
\newtheorem{thm}{Theorem}[section]
\newtheorem{lem}[thm]{Lemma}
\newtheorem{prop}[thm]{Proposition}
\newtheorem{cor}[thm]{Corollary}
\newcommand{\thmref}[1]{Theorem~\ref{#1}}
\newcommand{\lemref}[1]{Lemma~\ref{#1}}
\newcommand{\propref}[1]{Proposition~\ref{#1}}

\theoremstyle{definition}
\newtheorem{rmk}[thm]{Remark}

\newtheorem{conj}[thm]{Conjecture}

\newtheorem{qs}[thm]{Question}

\newcommand{\conjref}[1]{Conjecture~\ref{#1}}
\newcommand{\qsref}[1]{Question~\ref{#1}}

\newcommand*{\complex}{\mathbf{C}}
\newcommand*{\Q}{\mathbf{Q}}
\newcommand{\z}{\mf Z}
\newcommand{\C}{\mf C}
\newcommand{\h}{\mf H}
\newcommand{\re}{\mrm{Re}}
\newcommand{\im}{\mrm{Im}}

\newcommand{\mf}{\mathbf}
\newcommand{\q}{\quad}

\newcommand{\mc}{\mathcal}
\newcommand{\mk}{\mathfrak}
\newcommand{\mrm}{\mathrm}

\newcommand{\R}{\mf{R}}
\newcommand{\sltwo}{\mrm{SL}_2(\mf Z)}
\newcommand{\spn}{\mrm{Sp}_n(\mf Z)}

\newcommand{\GL}{\mrm{GL}}

\newcommand{\mcr}{\mathscr}
\newcommand{\g}{\Tilde{\gamma}}

\newcommand{\sltwor}{\mrm{SL}_2(\mf R)}

\newcommand{\spnr}{\mrm{Sp}_n( \mf R)}
\newcommand{\glnz}{\mrm{GL}_n( \mf Z)}
\newcommand{\glrz}{\mrm{GL}_r( \mf Z)}
\newcommand{\glnc}{\mrm{GL}_n(\mf C)}

\newcommand{\skgam}{S^n_k(\Gamma)}
\newcommand{\mkgam}{M_k(\Gamma)}

\newcommand{\mkngam}{M^n_k(\Gamma)}

\newcommand{\mnk}{M^n_k(\Gamma_0(N))}

\newcommand{\tr}{\mathrm{tr}\,}
\newcommand{\lan}{\langle }
\newcommand{\ran}{\rangle}

\newcommand{\sumn}{\sum \nolimits}

\newcommand*{\QEDB}{\hfill\ensuremath{\square}}
\newcommand{\smat}[4]{\left( \begin{smallmatrix}#1&#2\\#3&#4\end{smallmatrix}\right)}

\newcommand*{\norm}[1]{\left\lVert#1\right\rVert}

\newcommand{\psmb}{\left( \begin{smallmatrix}}
\newcommand{\psme}{ \end{smallmatrix} \right) }

\begin{document}

\title[Fourier coefficients and cuspidality]{Fourier coefficients and cuspidality for modular forms: another approach and beyond}

\author{Soumya Das}
\address{Department of Mathematics\\ 
Indian Institute of Science\\ 
Bangalore -- 560012, India \\.}
\email{soumya@iisc.ac.in}

\date{}

\begin{abstract}
We provide a simple and new induction based treatment of the problem of distinguishing cusp forms from the growth of the Fourier coefficients of modular forms. Our approach  gives the best possible ranges of the weights for this problem, and has wide adaptability. We propose a conjecture which asks the same converse question based on information on the Fourier-Jacobi coefficients, and answer it partially.
We also discuss how to recover cuspidality from the poles of the allied Rankin-Selberg $L$-series.
\end{abstract}

\maketitle 

\section{Introduction}

The main objective of this paper is to provide a short, simple and perhaps a new proof of the fact that suitable growth of Fourier coefficients of modular forms (at any given cusp) serve to distinguish between cusp forms and non-cusp forms. Whereas this statement is not very difficult to prove for elliptic modular forms (cf. \cite{ko}, \cite{bo-das1}, \cite{bo-ko}, or \propref{half-hecke}) it becomes more interesting -- and non-trivial for higher rank groups. Here we would be working with Siegel's modular group acting on Siegel's upper half-space, but most of the paper should generalize without much issue to other kinds of tube domains.

In \cite{ko}, Kohnen asked whether the Hecke bound distinguishes cusp forms in the context of elliptic modular forms of integral weights -- this question clearly is interesting only for the congruence subgroups. Over time, this question was investigated thoroughly -- and more generally, the following conjecture was formulated. Essentially it features two parameters. For notation appearing below, see Section~\ref{prelim}.

\begin{conj} \label{main-conj}
For $c \in \mf R$, a scalar weight $k_0 \in \mf N$, and $F \in \mkngam$ ($\Gamma \subset \spn$ is a congruence subgroup), suppose that for all $k \ge k_0$, the bound $a_F(T) \ll_F \det(T)^c$ holds for all $T \in \Lambda_n^+$ (see \ref{prelim} for the definition) at some cusp and for some $c<k-\frac{n+1}{2}$. Then $F \in \skgam$.
\end{conj}
It is clear that without loss one can assume that the cusp is $\infty$. The two parameters mentioned above are $c$ and $k_0$. Part of the question was also to find the largest value of $c$ and the smallest $k_0$ as mentioned in the above Conjecture~\ref{main-conj} which implies that $F$ is a cusp form. For the ``small weights'', i.e., when $k$ is close to $n/2$, the situation can be rather delicate.

This question is quite natural since often in practice, a modular form is defined via its Fourier expansion. Usually the theoretical tool to check cuspidality is the so-called Siegel's $\Phi$-operator. One needs to check that $F$ is in the kernel of such $\Phi$-operators at all the one-dimensional cusps. The strategy to answer the above question directly by decomposing the space in question in terms of Eisenstein series and cusp forms quickly turns very complicated in higher degrees owing to the presence of several Klingen-type Eisenstein series of various degrees and at various cusps. Their Fourier expansion is mostly not known in general, and from whatever is known, their Fourier coefficients grow almost similar to each other. This renders the direct approach described above as infeasible. The reader can refer to the introduction to \cite{bo-ko, bo-das1, bo-das2} for more comprehensive discussion on this topic.

Thus all investigations on this problem in higher degrees hinged on finding avenues avoiding the use of Klingen-Eisenstein series. To put things into perspective, we briefly sketch the evolution of the techniques that were employed on this problem, over the years.

-- (i) In \cite{ko, ko-ma, lino}, covering the cases $n=1,2$, the direct approach of using the newform theory and explicit Fourier expansions of elliptic and Jacobi respectively Eisenstein series respectively were used, focusing on Hecke'e bound in level one.

-- (ii) In \cite{bo-das1} the problem was reduced to eigenforms and Andrianov's identity was used to obtain information on the poles of the standard $L$-function $D_F(s)$ followed by Zahrkovskaya relations (which relate $D_F(s)$ with $D_{\Phi F}(s)$) to reduce the question to one on the cusp forms of lower degrees. The  results then followed by the comparison of the location of the poles of $D_F(s)$ via Shimura's results on its analytic properties. In fact there was another idea of writing $F$ as linear combination of theta series (via the solution to the basis problem of Eichler) for square-free levels and to use special properties of the theta series under the $\Phi$-operator. The cases covered (essentially) were $F\in\mnk$, $N$ square-free, $k >2n$, $c< k-n$.

-- (iii) In \cite{mizuno} a new idea, via Imai's converse theorem (in degree $2$, full level), was introduced into the problem. One proves that the various real-analytic Maa{\ss} twists of the Koecher-Maa{\ss} series of $F$ have good analytic properties from the hypothesis, and then applies Imai's converse theorem, focusing on the Hecke's bound.

-- (iv) In \cite{bo-das3} special techniques involving `diagonal' representatives of $0,1$-dimensional cusps of $\Gamma^{(n)}_0(N)$ were employed, reducing the problem to degrees $1$ or $2$ via the Witt-operators or the Fourier-Jacobi expansion, and appealing to the results from (ii) in degrees $1,2$. The cases covered (essentially) were $F\in M^n_k(\Gamma_0(N))$, $k \ge \frac{n}{2}+1$, $c< k-\frac{n+3}{2}$, $N \ge 1$; and also $k \ge \max\{ 2, \frac{n-1}{2} \}$, $c< k-\frac{n+1}{2}$, $N =1$ or square-free.

-- (v) In \cite{bo-das2}, the above procedure (cf. (ii)) was refined and a local version of the approach was used, leading to the comparison of the sizes of Hecke eigenvalues on a suitable set of primes. This allows one to ignore convergence properties of $D_F(s)$. The cases covered (essentially) were $F\in \mkngam$ ($\Gamma$ any congruence subgroup of $\spn$), $k >2n$, $c< k-n$.

-- (vi) In \cite{bo-ko}, a nice approach adapted from \cite{miyake} -- namely to transfer the problem to the behavior of $F(Z)$ as $\det(Y) \to 0$ -- was successfully carried out. The approach works smoothly when the Fourier expansion of $F$ is supported on positive definite matrices (at a cusp), but requires some additional tricks to reduce to this situation.
The cases covered (essentially) were $F\in\mkngam$, $k > n+1$, $c< k-\frac{n+1}{2}$.

We prove the following result, whose proof is direct, and which improves upon all of the the results described above in some aspect. Moreover, its proof is rather short, and it can be easily adapted to various other situations.

\begin{thm} \label{mainthm}
    Let $k \ge \frac{n}{2}$ be half-integral, $\Gamma \subseteq \spn$ be any congruence subgroup, and $F \in M^n_k(\Gamma)$. Suppose that for any $T \in \Lambda_n^+$ and for some $c< k - \frac{n+1}{2}$, one has the bounds $a_F(T) \ll \det(T)^c$ at the cusp $\infty$. Then $F$ is a cusp form.
\end{thm}

It is clear from \cite[Remark~5.3]{bo-das1} that the supremum $c_0$ of the possible values of $c$ in \thmref{mainthm} is indeed $\displaystyle c_0=k-\frac{n+1}{2}$. Now we talk about the quantity $k_0$ in \qsref{main-conj}. Whereas the genesis of the large values of $k_0$ in \cite{bo-das1} occurred from the poles of $D_F(s)$, in \cite{bo-ko} it was partly due to the use of the Lipschitz formula and partly due to the use of Hecke's bound in course of the proof. Since our argument will avoid all of these, we do not face any such issue. So our result is thus valid for the minimum possible value of $k_0$, viz., $k_0=n/2$ -- beyond which there do not exist non-zero cusp forms, by the theory of singular weights (cf. \cite{freitag}).

The main idea behind our proof is a suitable use of the Fourier-Jacobi (FJ) expansion, followed by Miyake's approach as adopted in \cite{bo-ko}. It is perhaps well-known that the FJ-expansions serve as a very useful  and effective  tool in tackling a variety of questions about Siegel modular forms in higher degrees. Our inductive argument reduces the question to those $F$ having a Fourier-expansion supported on positive-definite coefficients -- from which the final result follows -- our approach on this route is considerably shorter than that in \cite{bo-ko}, and it deals with bounds only for elliptic modular forms. 

Another technical point is that in order for our argument to succeed, even only for integral weights, we have to crucially work with half-integral weights and general congruence subgroups right from the beginning, unlike all of the other works quoted above, where such weights were not considered as they were not necessary in the proof. This is because when one works with the various (corresponding to each maximal parabolic subgroup) Fourier-Jacobi expansions, one encounters Jacobi forms whose theta components may have half-integral weights. It is however noted in some places that some of the earlier works generalize to half-integral weights with further technical input. 

In some sense and in view of the above, our proof may be in retrospect somewhat similar in spirit to the arguments in \cite{bo-das4}; however the situations are completely different. One can ask  converse statements -- akin to that in \thmref{mainthm} -- for the various Fourier-Jacobi (FJ) expansions of $F$ as well. If $F$ is cuspidal of course all the FJ-coefficients $\phi_{F, \mc M}$ (see \eqref{fj-exp}) are cuspidal for all $\mc M>0$ of size $r$ with $1 \le r \le n-1$. The question therefore is the converse to this. This theme seems quite interesting, and was considered in \cite{bo-das1} in level one. Here we want to consider higher levels, and propose the following.

\begin{conj} \label{fj-conj}
    Let $1 \le r \le n-1$ and $F \in M_k^n(\Gamma)$. Suppose that for any fixed $r$ as above, and for all but finitely many indices $\mc M \in \glrz  \backslash \Lambda_r$, the
    FJ-coefficients $\phi_{F, \mc M}$ are cuspidal. Then $F$ is cuspidal.
\end{conj}
We prove this conjecture for $\Gamma=\Gamma^n_0(N)$ when $r=n-1$ by utilizing the special double coset representatives in $\Gamma^n_0(N) \backslash \mrm{Sp}_n(\Q)/ P_{n,n-1}(\Q)$, where $P_{n,n-1}(\Q)$ is a Klingen-parabolic, which is relevant in checking the cuspidality property via the Siegel $\Phi$ operator, see \eqref{pnr}. Finally let us mention that \conjref{fj-conj} may have some bearing with the questions on `formal Fourier-Jacobi expansions' of Siegel modular forms, see e.g. \cite[Proof of Theorem~9.2]{ibu-formal}.

In subsection~\ref{poles}, we discuss how the poles of the Rankin-Selberg $L$-series $\mc R(F,s)$ of $F \in \mkngam$ control its cuspidality. We show that if $k>n+1$, then $F$ is cuspidal if and only if $\mc R(F,s)$ does not have certain poles, cf. \propref{propole}. 

We end the introduction by remarking that the method seemingly can not handle vector-valued Siegel modular forms because there is no simple Fourier-Jacobi expansion available in this setting. This is at least known to the experts working in this area. One might try other approaches, see subsection~\ref{vvsmf} for a discussion on this.  We also discuss a refinement of \ref{main-conj} on certain `thin-sets', see subsection~\ref{thin}. Finally let us mention that in subsection~\ref{ext-fj} we propose a question which asks to recover cuspidality from a suitable growth of the extended Petersson norms of the FJ coefficients.

\subsection*{Acknowledgements}
{\small
We thank Prof. T. Ibukiyama for his comments and for drawing the author's attention to the works \cite{Ibu-SK} and \cite{ibu-formal}.
We thank the referee for a very meticulous reading of the paper and for many thoughtful suggestions which improved the paper.
The author thanks IISc. Bangalore, UGC Centre for Advanced Studies, DST India for financial support.
}

\section{Notation and setting} \label{prelim}
We will use standard notation throughout the paper. We use $\z, \mf N, \Q, \mf R$, and $\complex$ to denote the integers, positive integers, 
rationals, reals, and complex numbers, respectively. We put $A[B]:=B^tAB$ for appropriate matrices $A,B$.


For basic facts about Siegel modular forms we refer to \cite{andrianov2}, \cite{freitag} or 
\cite{klingen1}. The symplectic group $\spnr$ defined as
\[ \spnr = \{ g= \smat{A}{B}{C}{D} \in M({2n},\R) \mid M^t \smat{0_n}{-1_n}{1_n}{0_n}M= \smat{0_n}{-1_n}{1_n}{0_n} \} ,\]
acts on Siegel's half-space $\mf H_n$ defined as 
 \[  { \mf H_n := \{ Z= Z^t \in M(n, \mf C) \mid (Z - \overline{Z})/2i \text{ is positive definite} \}, } \] 
in the usual way by $g \langle Z \rangle=(AZ+B)(CZ+D)^{-1}$ and
on functions $F:{\mf H}_n\longrightarrow {\mf C}$ by the {weight $k \in \z$, $k \ge 0$} ``stroke'' operator:
\[ (F\mid_kg)(Z)= \det(CZ+D)^{-k} F(g \langle Z \rangle) \qquad  (g=\left(
\begin{smallmatrix}
A & B\\ C & D\end{smallmatrix}\right)\in \mrm{Sp}(n, \mf R), k \in \mf Z_{\geq0} ). \] 
When $k \in \frac{1}{2} \z$ and $k \ge 0$, one has to define the stroke operator and the setting somewhat differently (via double covering of $\spnr$), we refer the reader to \cite[(4.37)~Chapter~1]{andrianov2}.

Throughout the paper we put $\Gamma_n=\spn=\spnr \cap M(2n, \z)$.
For $g \in \Gamma_n$, we also write $g= \psmb a_g & b_g \\ c_g & d_g \psme$. 
We write $\displaystyle J(g,Z) := \det(c_g Z+d_g)$ as the automorphy factor. We also put $\h:= \mf H_1$.

A holomorphic function $F$ on
${\mf H}_n$ is called a modular form for $\Gamma$ of weight $k \in \z$, $k \ge 0$, if it satisfies the transformation law
\[ F\mid_k\gamma= F \qquad (\gamma=\left(\begin{smallmatrix}
A & B \\ C & D\end{smallmatrix}\right)\in \Gamma), \]
with the additional condition of being holomorphic at the cusps when $n=1$.
We denote the space of all such functions by $M^n_k(\Gamma)$. When $k$ is a half-integer, similar definitions apply, and a detailed account can be found in \cite[Chapter~2, \S~2]{andrianov2}.

Every $F \in \mkngam$ can be expressed by a Fourier series of the form
\begin{equation} \label{ffe}
    F(Z) = \sumn_{T \in \Lambda_n } a_F(T) e(TZ/M);
\end{equation}
for some integer $M \ge 1$. Here $\Lambda_n$ denotes the set of half-integral, positive semi-definite matrices; and $\Lambda_n^+$ its subset consisting of positive-definite matrices. Now $F$ is a cusp form if and only if \eqref{ffe} is supported on $\Lambda_n^+$ at all the cusps of $\Gamma$. We denote the space of cusp forms on $\Gamma$ by $\skgam$.
We put $\Gamma_n=\spn$, and put $M^n_k=M^n_k(\Gamma_n)$, omit the superscript $n$ if $n=1$, etc.

Let $N$ be the level of $\Gamma$. For any $0 \le r \le n-1$, $F \in \mkngam$ has a Fourier-Jacobi expansion for $F$ (of type $(n-r,r)$), and $\phi_{F,\mc M}$ being the Fourier-Jacobi (FJ) coefficients of $F$:
\begin{align} \label{fj-exp}
    F(Z) = \sumn_{\mc M \in \Lambda_r} \phi_{F,\mc M}(\tau,z) e(\mc M \tau'/N)
\end{align}
where $Z= \psmb \tau & z \\ z^t & \tau' \psme$, with $\tau \in \mf H_{n-r}, \tau' \in \mf H_r$; $\mf H_n$ being the Siegel's upper half-space of degree $n$. It is well-known that $\phi_{F,\mc M} \in J_{k, \mc M}$, the space of Jacobi forms of weight $k$ and index $\mc M$, see e.g. \cite{ziegler} for more details.

It is also known that $F$ as above is cupidal if and only if $F$ is in the kernel of the Siegel-$\phi$ operator at all the cusps -- which again is equivalent to checking that $\Phi(F|\gamma)=0$ for all $\Gamma \subseteq \Gamma \backslash \Gamma_n / P_{n,n-1}(\Q)$. We recall that for any $1 \le r \le n-1$ and any  commutative subring $R$ of $\R$; the following parabolic subgroups of $\mrm{Sp}_n(R)$ are defined by: 
  \begin{align} \label{pnr}
   P_{n,r}(R):= \{ \begin{pmatrix}
    a_{11} & 0 & b_{11} & b_{12} \\
    a_{21} & a_{22} & b_{21} & b_{22}\\
    c_{11} & 0 & d_{11} & d_{12}\\
    0 & 0 & 0 & d_{22}
    \end{pmatrix} \in \mrm{Sp}_n(R),  \q *_{11} \text{ is } r \times r, \, \, *_{22} \text{ is } (n-r) \times (n-r) \} .
\end{align}
Moreover, $\Phi (F)(Z_1) = \lim_{t \to \infty} F\psmb Z_1 & 0 \\ 0 & it \psme$.

We further let (with $A^t$ denoting the transpose of $A$ and $A^{-t}:=(A^{-1})^t$ for invertible $A$)
\begin{equation}
    \mathscr U(\Gamma)=\{  \psmb U & 0 \\ 0 & U^{-t} \psme \in \Gamma \mid U \in \glnz \}
\end{equation}
be the subgroup of the embedded copy of $\glnz$ inside $\Gamma$.

\subsubsection{Vector-tuple modular forms} \label{vvmf-sec}
An $n$-tuple $\mf f:= (f_1,\ldots,f_d)^t$ of holomorphic functions on $\mf H$ is called a vector-tuple modular form (v.v.m.f., to be distinguished from the usual vector-valued modular forms, see subsection~\ref{vvsmf}) of weight $k \in \tfrac{1}{2}\mf Z$ with respect to a representation $\rho \colon \widetilde{\Gamma} \to GL(d,\mf C)$ on a congruence subgroup $\widetilde{\Gamma}$ of the metaplectic double cover $\widetilde{\spn}$ of $\spn$, if
\[ \mf f \mid_k \g(Z)= \rho(\g) \mf f(Z), \q (Z \in \h_n, \g \in \widetilde{\Gamma}), \]
and if $n=1$, we impose the condition that $\mf f$ remains bounded as $\Im(Z) \to \infty$. The action $\mid_k$ operation is defined component-wise on $\mf f$. Let us denote the space of such functions by $\widetilde{M}_n(k,\rho, \widetilde{\Gamma})$. Such an $\mf f$ has a Fourier expansion of the form
\begin{align} \label{fe-vvmf}
    \mf f (Z) = \sumn_{T \in \Lambda_n} \overset{\rightarrow}{a}_{\mf f}(T) e(TZ/M),
\end{align}
where $\overset{\rightarrow}{a}_{\mf f}(n) \in \C^d$.
Further, the spaces of cusp forms $\widetilde{S}_n(k,\rho, \widetilde{\Gamma})$ consists of those $\mf f \in \widetilde{M}_n(k,\rho, \widetilde{\Gamma})$ whose Fourier coefficients (as in \eqref{fe-vvmf}) at all cusps are supported on positive-definite matrices.

\section{The case of half-integral weights when \texorpdfstring{$n=1$}{} } \label{appn}

\textsl{We let $k \in \tfrac{1}{2} \mf Z$ unless stated otherwise}. 

In the Proposition below we prove our main result of the paper for $n=1$, for all possible half-integral weights and for all congruence subgroups. We use Miyake's characterization \cite[Theorem 2.1.4]{miyake}  of cusp forms given below; which is however stated in loc. cit. only for integral weights.

\begin{thm} \label{miyake}
    Suppose $\mk f \in \mkgam$ be such that $\mk f (z) = O(y^{-\nu})$ as $y=\im(z) \to 0^+$ for some $0<\nu<k$ holds uniformly in $x=\re(z)$. Then $f$ is a cusp form.
\end{thm}
We will use this when $k$ is integral. Alongside, we will also offer an independent way to tackle half-integral weights. The main result of this section is the following. It is a purely function-theoretic result, and has nothing to do with arithmetic. The result is certainly expected (cf. \cite[\S 5,(i)]{bo-ko}), but we could not find this result in the literature.

\begin{prop} \label{half-hecke}
Let $ k \geq 1/2$ be half-integral and $\mk f \in M_k(\Gamma)$ be non-zero. If $|a_{\mk f}(n)| \ll_{\mk f} n^{c}$ for all $n \ge 1$, and for any fixed $ c < k-1$, then $\mk f$ must be a cusp form.
\end{prop}

\begin{proof}
Our proof is based on the argument in \cite{miyake}. First of all  from the Fourier expansion \eqref{ffe} of $\mk f$, and the bound in the lemma, we find, bounding absolutely, that (with any $\beta >0$)
\begin{align} \label{fe1}
    \mk f(z) \ll_{\mk f,  \beta} |a_{\mk f}(0)| + \sumn_{n \ge 1} n^c \exp(- \frac{2 \pi}{L} \cdot ny)
    \ll |a_{\mk f}(0)| + \sumn_{n \ge 1} n^c \cdot (ny)^{-\beta},
\end{align}
where $L$ is the level of $\Gamma$ and the implied constants depends only on $\mk f$. We avoided using Lipschitz's formula here.

Let us now choose $\epsilon>0$ such that $c+\epsilon<k-1$.
We first assume that $c \ge -1$. Simply using that $e^x \gg x^{1+c+\epsilon}$ ($x>0$), \eqref{fe1} shows that, for all $0<y<1$, the bound $\displaystyle \mk f(z) \ll y^{-(1+c+\epsilon)}$.
Now if $c \le -1$, we simply use $n^c \le n^{-1}$ in \eqref{fe1} and get in this case, the bound $\displaystyle \mk f(z) \ll y^{-\epsilon}$. Summarizing,
\begin{align} \label{fe2}
    \mk f(z) &\ll y^{-(1+c+\epsilon)} & \text{ if } c \ge -1, \\
    \mk f(z) &\ll y^{-\epsilon} & \text{ if } c \le -1.
\end{align}
These bounds show that there exists $0<\nu<k$ such that (uniformly in $x$ and for all $0<y<1$)
\begin{align} \label{fe3}
    \mk f(z) &\ll y^{-\nu}.
\end{align}

We now invoke \thmref{miyake}.
Even though Miyake states the above result for $k$ \textit{integral}, one may  note that it is valid for any \textit{real} weight $k >0$ with suitable modification of the setting.
We however choose to invoke it only for integral weights.

For this, simply consider $\mk g(z)=\mk f(z)^4$. Then $\mk g \in M_{k}(\Gamma)$ and $k=4k \ge 2$, with $k$ integral. Then clearly we have $\mk g(z) \ll y^{- \eta}$, with $0<\eta<k$ as $y \to 0$. Therefore by Theorem~\ref{miyake}, $\mk  g= \mk f^4$ is a cusp form and hence so is $\mk f$, which follows immediately by writing the definition of $\mk f|\gamma$ for $\Gamma \subseteq \sltwo$ corresponding to the cusps. The lemma is thus proved.
\end{proof}

\begin{rmk}[The vector-tuple case]
We can also easily prove the following:
   Let $ k \geq 1/2$ and  $\mf f \in \widetilde{M}_1(k,\rho)$  be non-zero. If $\| \overset{\rightarrow}{a}_{\mk f}(n) \| \ll_{\mk f} n^{c}$ for all $n \ge 1$, and for any fixed $ c < k-1$, then $\mf f$ must be a cusp form.

Here, we denote by $\| \cdot \|$ the usual $L^2$ norm on $\mf C^d$.
The proof is verbatim the same as that of \propref{half-hecke}, if we note that
\[  \| \mf f(z) \| \le \sumn_{n \ge 0} \| \overset{\rightarrow}{a}_{\mk f}(n) \|  \exp(- \alpha \cdot ny) \ll_{\mk f, \alpha, \beta} \| \overset{\rightarrow}{a}_{\mf f}(0) \| + \sumn_{n \ge 1} n^c \exp(- \alpha \cdot ny),\]
for some constant $\alpha>0$, after which everything works as before, if we further note that $\mf f \cdot \theta = (f_1 \theta, f_2\theta, \ldots)$ is a v.v.m.f. for another representation on a possibly smaller congruence subgroup.
\end{rmk}

If $\rho$ factors through a congruence subgroup of $\widetilde{\spn}$ (the metaplectic double cover of $\spn$), then one can easily show that the respective FJ-coefficients are scalar valued Jacobi forms and the analogue of the above Remark holds in higher degrees (i.e. in the true vector-valued case), cf. the next Section and Section~\ref{vvsmf}. However the same is not clear for a general $\rho$, even though we expect a positive answer.

\begin{conj}
    The analogue of \thmref{mainthm} holds for the space $\widetilde{M}_n(k,\rho, \widetilde{\Gamma})$ for any given $\rho$.
\end{conj}

\section{The case of higher degrees}
We would now generalize the  results of the previous section to the higher degrees, i.e., when $n>1$. The main tool will be the Fourier-Jacobi expansion along with an inductive argument on the degree $n$. The crucial observation would be that the maximal exponent in the growth condition, viz. $\displaystyle k-\frac{n+1}{2}$, is consistent along each Fourier-Jacobi expansion of $F$.

We have to consider the following theta decomposition for each $\phi_{F, \mc M}$ given by:
\begin{align} \label{theta-decomp}
    \phi(\tau,z) = \sumn_{\mu \in  (2 \mc M)^{-1}\mf Z^{r,n-r} \big/ \, \mf Z^{r,n-r} } h_\mu(\tau) \theta_{\mc M,\mu}(\tau,z). 
\end{align}
Then it is known (cf. \cite[Corollary~2 to Theorem~3.3]{ziegler}) that $h_\mu \in M_{k - \frac{r}{2}}(\Gamma')$ for some congruence subgroup $\Gamma'$. Note that the theta series has weight $\frac{r}{2}$ in this case. 
Let $F \in \mkngam$ with the FJ expansion from \eqref{fj-exp}: 
$\displaystyle F(Z) = \sumn_{\mc M \ge 0} \phi_{F,\mc M}(\tau,z) e(\mc M \tau'/N)$.
Then the Fourier expansion of the theta components $h_\mu$ of $\phi_{F,\mc M} $ are given by (loc. cit.)
\begin{align} \label{fe-hmu}
    h_\mu(\tau) = e(- \frac{1}{4} \mc M^{-1}[\mu] \tau/N) \sum_{ l \in  \Lambda_{n-r}, \, l \ge \frac{1}{4} \mc M^{-1}[\mu]  } c(l, \mu) e(l \tau/N),
\end{align}
where 
\begin{align} \label{clmu}
    c(l, \mu) := a_F (\begin{pmatrix} l & \mu/2 \\ \mu^t/2 & \mc  M \end{pmatrix}).
\end{align}

\begin{lem} \label{cusp-theta-comp}
    Let $\mc M \in \Lambda_r^+$, $k \in \tfrac{1}{2} \mf Z$, and $\phi_{\mc M} \in J_{k,\mc M}(\Gamma)$. Then $\phi_{\mc M}$ is cuspidal if and only if all its theta components $h_\mu$ are cuspidal.
\end{lem}

\begin{proof}
If $\phi_{\mc M}$ is cuspidal, then the same is true for all $h_\mu$ by looking at their Fourier expansions cf. \eqref{fe-hmu}, \eqref{clmu} at all the cusps. The $h_\mu$ inherit their Fourier expansion from that of $\phi_{\mc M}$ at all the cusps.

Conversely, suppose that all the $h_\mu$ are cuspidal. We check that for each $\Gamma \subseteq \Gamma_n$, the Fourier expansion of $\phi_{\mc M}$ is supported on positive definite matrices. To see this, note that $\theta_{\mc M,\mu}|\gamma $ is again a linear combination of the same set of $\theta_{\mc M,\nu}$ (via the action of the Weil representation), and that the Fourier coefficients of $\theta_{\mc M,\mu}$ are supported on non-negative matrices. The proof thus follows from \eqref{theta-decomp}, with $\phi_{\mc M}$ replaced by $\phi_{\mc M}|\gamma$.
\end{proof}

\begin{lem} \label{lipschitz-lem}
    For $Y>0$, real numbers $\alpha, \beta, \epsilon >0$ one has the bound
    \begin{align}
        \sumn_{T \in \Lambda_n^+} \det(T)^\alpha \exp(- \beta (\tr T Y)) \ll_{\alpha, \beta, n, \epsilon} \det(Y)^{- (\alpha + \frac{n+1}{2} +\epsilon) }.
    \end{align}
\end{lem}

\begin{proof}
    A proof in principle can be obtained from \cite[Proposition~3.7 or Lemma ~6.2]{das-krishna} with some observations, but here we present a simpler proof. Since we don't have to track the dependence on $\alpha$, the job is easier. Let us denote the LHS above by $\mc S(Y)$. Then clearly $\mc S(Y) = \mc S(Y[U])$ for any $U \in \glnz$ by absolute convergence, and thus we can assume that $Y$ is Minkowski-reduced. This implies in particular that (see e.g. \cite[p~.20]{klingen1}) 
    \[ c_1 \ Y^* \le Y \le c_2 \ Y^*; \]
$Y^*$ denotes the diagonal matrix with diagonal entries of $Y$, and $c_1,c_2>0$ are constants depending only on $n$.

    Put $a_i$ to be the $i$-th diagonal element of a matrix $A$. Therefore,
    \begin{align}
        \mc S(Y) & \ll \sumn_{T} \det(T)^\alpha \exp(- \beta (\sumn_{i=1}^n t_i y_i)) \\
        & \ll \sumn_{t_1,t_2,\ldots t_n} \mc A(t_1,t_2,\ldots t_n) (\prod\nolimits_i t_i)^\alpha \exp(- \beta (\sumn_{i=1}^n t_i y_i)) \label{sy-bd}
    \end{align}
where $\displaystyle \mc A(t_1,t_2,\ldots t_n) = \{ S \in \Lambda_n^+ \mid s_i=t_i, i=1,2,\ldots,n\}$. From the inequalities $s_{ij} \le 2 \sqrt{s_i s_j}=2 \sqrt{t_i t_j}$ with $i \neq j$ we immediately get that 
\[\# \mc A(t_1,t_2,\ldots t_n) \ll (\prod\nolimits_i t_i)^{(n-1)/2}.\] Thus the variables are de-coupled and we are reduced to the one-dimensional problem:
\begin{align} \label{n=1 bd}
     \sumn_{L \ge 1} L^c \exp(- 2 \pi L v) \ll_{c,\epsilon} v^{-(c+1+\epsilon)} \q (v >0);
\end{align}
for every $\epsilon>0$. To see this elementary estimate, simply use the argument as in \eqref{fe1}.

Thus from \eqref{sy-bd} and \eqref{n=1 bd} we obtain
\begin{align}
     \mc S(Y) & \ll \prod \nolimits_i \left( \sumn_{t_i}  t_i^{ \alpha+\frac{n-1}{2} } \exp(- \beta \, t_i y_i) \right)  \ll  \prod\nolimits_i y_i^{-( \alpha+\frac{n+1}{2} +\epsilon)} \ll \det(Y)^{-( \alpha+\frac{n+1}{2} +\epsilon)},
\end{align}
since $Y$ is reduced. The Lemma follows.
\end{proof}

We will now prove \thmref{mainthm}.
 
\subsection{Proof of \thmref{mainthm} }
We induct on $n$. For $n=1$, this is \lemref{half-hecke}. Next, we assume the theorem for all congruence subgroups in degree $n-1$. 
Our idea is to pass to degree $n$ by using the Fourier-Jacobi expansion for $F$ (of type $(n-1,1)$) from \eqref{fj-exp}.

\subsection{Claim} \label{claim}
We next would prove that for all $\mc M \in \Lambda_r^+$ ($1 \le r \le n$), the $\phi_{F, \mc M}$ are cuspidal and that $\phi_{F,0}=0$. 
The Claim would be proven in course of the next two subsections. 

\subsection{\texorpdfstring{$\phi_{F, \mc M}$ are cuspidal for $\mc M>0$}{a}}
To prove that $\phi_{F, \mc M}$ is cuspidal, we can consider $\mc M \in \Lambda_r^+$ to be fixed. Referring to the theta decomposition (cf. \eqref{theta-decomp}, of $\phi_{F, \mc M}$ we would prove that for each $\mu$, $h_\mu$ is cuspidal. This being granted, Lemma~\ref{cusp-theta-comp} immediately shows that $\phi_{F, \mc M}$ is cuspidal.

It remains to deal with the $h_\mu$'s. Note that $\mu$ is also fixed here. Put $\displaystyle m:= \det(\mc M)$. We consider $\displaystyle g(\tau):=h_\mu(4m \tau)$ ($\mu \in \mf Z^{r,n-r}$, cf. \eqref{fe-hmu}) so that its Fourier expansion may be written as 
\begin{align} \label{g-fe}
    g(\tau) =  \sum_{ L \in \mf \Lambda_{n-r}, \, L \equiv - m(\mc M)^{-1}[\mu] \bmod{4m} } c \left(\frac{L}{4m} +  (\mc M)^{-1}[\frac{\mu}{2}] , \mu \right) e(L \tau)
\end{align}
We look at the positive definite Fourier coefficients of $g$.
Upon using the Jacobi decomposition for positive definite matrices, from \eqref{clmu} and \eqref{g-fe} we can write by our hypothesis on $F$ that,
\begin{align} \label{agL-coeff}
    a_g(L) \ll \det \left( \begin{pmatrix}
        \frac{L}{4m} + (\mc M)^{-1}[\mu/2]   & \mu/2 \\ \mu^t/2 & \mc M 
    \end{pmatrix} \right)^c =  \det(L)^c .
\end{align}
since $m$ is fixed, and $c$ is as in the theorem. 

Put $k_r:=k - \frac{r}{2}$. Now notice the identity for $1 \le r \le n-1$:
    \begin{align}
     k-\frac{n+1}{2}= k - \frac{r}{2} - \frac{n-r+1}{2} = k_r - \frac{(n-r)+1}{2} .
\end{align}
    Since $\phi_{F, \mc M}$ has weight $k$, and its theta components have degree $n-r$ (cf. \cite[\S 3]{ziegler}) and weight $k - \frac{r}{2}$, the weight of $g$ is also $k_r$. 
 We can therefore invoke the induction hypothesis on $g$ to conclude that all the $h_\mu$ are cusp forms, as desired.

\subsection{\texorpdfstring{$\phi_{F,0}=0$}{a}} \label{M not pos-def}
We first consider the peripheral FJ-expansion of the type $(n-1,1)$ of $F$ and look at $\phi_{F,  m}$ with $ m \ge 0$ being an ordinary integer.
We want to show that $\phi_{F,0}=0$. Towards this, we use an argument similar to that in \cite[Theorem~5.7]{bo-das1}.
We  apply $\tilde \Phi$ -- the \lq opposite\rq \ $\Phi$-operator to $F$, which is defined by
\[ f(w)=\tilde \Phi F (w):=\lim_{\lambda\to\infty} F \left( \left( \begin{smallmatrix}  i\lambda & 0\\
0 & w\end{smallmatrix}\right)  \right) \q (w\in {\mf H}_{n-1}). \]
$\tilde \Phi $ has the same property as that of the usual $\Phi$ operator; in particular the Fourier coefficients of $\tilde \Phi F$ are supported on the lower right $n-1 \times n-1$ blocks of matrices in $\Lambda_n$. Then by invoking the induction hypothesis as shown in the previous subsection -- the Fourier-Jacobi coefficients $\phi_{F,  m}(\tau,z)$  with $ m \ge 1$ being cuspidal -- all go to zero under $\tilde \Phi F$, leaving with us the equation
\begin{align} \label{phi00}
    f(w)= \phi_{F,0}(\tau_1,z_1) e(0 \cdot \tau_1') + 0 \cdot e( \tau_1')  + 0 \cdot e(2 \tau_1') +\ldots \q (w= \psmb \tau_1 & z_1 \\ z_1^t & \tau_1'\psme \in \h_{n-1}).
\end{align}
If $\phi_{F,0} \neq 0$, this shows by Weissauer's result on singular modular forms that $f$ is singular -- its Fourier expansion is supported on matrices with right lower corner entry $0$. Note that the property of singularity can be checked at any cusp, as the condition is only that \lq $\text{weight } < \tfrac{1}{2}\text{ degree}$ \rq. This is however a contradiction since $k > \frac{n-1}{2}$. Therefore we must have $\phi_{F,0}=0$, proving the {\bf Claim}.

\subsection{Conclusion of the proof of \thmref{mainthm}}
 We first note  that the {\bf Claim} (applied to  the FJ-expansion of the type $(n-1,1)$ shows that the Fourier expansion of $F$ (at $\infty$) is supported on positive-definite matrices. This information will now be used to finish the proof of \thmref{mainthm}.

To prove the theorem, we first apply Lemma~\ref{lipschitz-lem} to the (necessarily positive-definitely supported) Fourier expansion of $F$. We get
\begin{align}
    F(Z) \ll \det(Y)^{-( c+\frac{n+1}{2} )} \q (Z \in \h_n).
\end{align}
This part of the proof is similar to that in \cite{bo-ko}, but a bit simpler.
Without loss, we can, and assume that the width of the cusp at $\infty$ is $1$.
We then choose $\Delta \subset \Gamma_n$ with the property that $\Delta$ is a complete set of right coset representatives of $\Gamma \subseteq \Gamma_n$ such that for each $\delta = \psmb a_\delta & b_\delta \\ c_\delta & d_\delta\psme\in \Delta$, one has $c_\delta $ is non-singular. This is standard, see eg. \cite[Lemma~1]{bo-ko} or \cite[Proof of Cor.~5.2]{das-supnorm}. Let us fix a $\delta$ as above.

We consider the Fourier coefficient $a_\delta(T)$ of $F_\delta=F | \delta$. Assume that $T$ is singular. We want to show that $a_\delta(T)=0$. From the equation
\begin{align}
a_{\delta}(T) = \det(U)^{-k} a_{\delta \psmb U & 0 \\ 0 & U^{-t} \psme }({T[U]}),
\end{align}
it is evident that we can assume without loss that $T= \psmb T_1 & 0 \\ 0 & 0 \psme$, where $T_1$ is of size $n-1$ -- by replacing $\delta$ with $\delta \psmb U & 0 \\ 0 & U^{-t} \psme$ and $T[U]$ with $T$. That there exist an $U \in \glnz$ such that $T[U]=\psmb T_1 & 0 \\ 0 & 0 \psme$ is well-known, see e.g. \cite[Lemma~3.14]{andrianov2}.

From this, we proceed in the usual manner: for any fixed $Y_0>0$ and any $\epsilon>0$, 
\begin{align}
    a_{\delta}(T) &= \exp(- 2 \pi \tr (T Y_0)) \, \int_{X \bmod 1} F_\delta(X+iY_0) e(-TX)dX \\
    & \ll_\epsilon \exp(- 2 \pi \tr (T Y_0)) \det(Y_0)^{-k+( c+\frac{n+1}{2} +\epsilon)} 
\end{align}
since 
\[ \det( \im(\delta(Z))) = \det(Y_0) \big( \det(c_\delta) |\det(Z+c_\delta^{-1}d_\delta)|\big)^{-2} \gg \det(Y_0)^{-1},\]
and 
\[ |F_\delta(Z)|\le |J(\delta,Z)|^{-k} |f(Z)| \ll_\epsilon \det( \im(\delta(Z)))^{-k}  \det(Y_0)^{-( c+\frac{n+1}{2} +\epsilon)}.\] 

We now choose $Y_0 = \psmb 1_{n-1} & 0 \\ 0 & t \psme$, where $t>0$. Then $\tr (T Y_0)= \tr(T_1)$, which is fixed and independent of $t$, and $\det(Y_0)=t$. 

Next choose $\epsilon_0>0$ small enough so that $c+\epsilon_0<k-\frac{n+1}{2}$, and take $\epsilon:=\epsilon_0$.

Now letting $t \to \infty$, we see that $a_{\delta}(T) =0$, as desired. Therefore $F$ is cuspidal, and the proof is complete.
\QEDB

\begin{rmk}
    We note that one could have used any of the Fourier-Jacobi expansions of $F$ (w.r.t.) the different parabolic subgroups, and the same argument as above would have worked, showing that $\phi_{F, \mc M}=0$ for all non-singular $\mc M$ of size $1 < r \le n$. However, we could not have invoked the \cite[Theorem 5.7]{bo-das1}, since the it requires that for all but \textit{finitely many} $M \in \glrz \backslash \Lambda_r$, $\phi_{F, \mc M}$ be cuspidal. Unfortunately, a priori we do not have this information for $r>1$.
\end{rmk}

\begin{rmk}
    We note that the growth of the Fourier coefficients and that of $F$ as $\det(Y) \to 0$ are essentially equivalent. One way is the content of \cite{bo-ko} (cf. \cite{miyake}), and the other way follows immediately from the formula of $\displaystyle a_F(T) = \int_{X \bmod 1, Y=T^{-1}} F(Z)e(-TZ)dX$. This suggests a different approach -- of bounding $F$ suitably for `small' $Y$ -- towards bounding Fourier coefficients of cusp forms.
\end{rmk}

\subsection{Cuspidaliy from Fourier-Jacobi coefficients} \label{fj-sec}
This subsection explores the relationship between a Siegel modular form and its various FJ-expansions in terms of cuspidality. We will prove the Conjecture~\ref{fj-conj} in the case $\Gamma= \Gamma^n_0(N)$ and $r=n-1$.

We first prove a lemma, which generalizes  \cite[Lemma~3.4]{Ibu-SK}, and is inspired by it. We define the matrix $R$ by $R= \psmb M & 0 \\ 0 & M \psme$, where $M $ is the $n \times n$ matrix with $1$'s on the anti-diagonal, and $0$'s elsewhere. 

\begin{lem} \label{pnn-1}
    The representatives of the double cosets $\Gamma^n_0(N) \backslash \mrm{Sp}_n(\Q)/ P_{n,n-1}(\Q)$ can be chosen from the set $P_{n,n-1}(\Q)R$ where $R$ is defined as above.
\end{lem}

\begin{proof}
The proof is modeled on the one which is presented in \cite[Lemma~3.4]{Ibu-SK}, and we will be brief. We try to use the same notation as loc. cit. So we put $P'=RP_{n,n-1}(\Q)R$ and note that its enough to prove that the representatives can be chosen from $P_{n,n-1}(\Q)$.

First of all, we choose coset representatives in $g \in \Gamma^n_0(N) \backslash \mrm{Sp}_n(\Q)$ with $c_g$ non-singular. Next we multiply $g$ from the left with $\psmb U & 0 \\ 0 & U^{-t} \psme$ such that $U \in \glnz$, to assume that $C$ is upper triangular. We write $C= \psmb c_{11} & c_{12} \\ c_{21} & c_{22} \psme$ with $c_{11}$ being $(n-1) \times (n-1)$ and $c_{22}$ being $1\times 1$. We make the same convention for other $n \times n$ blocks as well.

Next, we apply multiply the resulting $g$ on the right by the matrix $1_{n-1} \times V \in P'$ where $V \in \mrm{GL}_2(\z)$. If $\displaystyle V=\psmb a & b\\c & d \psme$, then $1_{n-1} \times V$ is defined as
\[
1_{n-1} \times V :=\begin{pmatrix}
            1_{n-r}&0&0&0\\
            0&a&0&b\\
            0&0&1_{n-r}&0\\
            0&c&0&d
    \end{pmatrix}.
    \]
This results in 
\begin{align}
    C= \begin{pmatrix}
        c_{11} & * \\ 0 & c_{22} v_{11} + d_{22} v_{21} \end{pmatrix} , \q D= \begin{pmatrix}  * & * \\ * & c_{22} v_{12} + d_{22} v_{22} \end{pmatrix} .
\end{align}
We choose $V$ such that $(c_{22} v_{11} + d_{22} v_{21}, c_{22} v_{12} + d_{22} v_{22}) = (c_{22}, d_{22})V=(0,*)$. Thus we can assume that $c_{22}=0$. Thus the last row of $C$ is $0$.
Notice that $c_{11}$ is upper triangular and non-singular. Thus we can multiply the resulting $g$ from the right by the matrix $\psmb B & 0 \\ 0 & B^{-t} \psme \in P' $
where $\displaystyle B= \psmb 1_{n-1}  & - c_{11}^{-1}c_{12} \\ 0 & 1 \psme$. Here we have continued to use the same notation for the successively modified matrices $C,D$ etc. This ensures that $c_{12}=0$. We thus obtain $C= \psmb c_{11} & 0 \\ 0 & 0\psme$. The rest of the argument which shows that $d_{21}=0$ and $a_{21}=0$ remain the same as in the case $n=2$ and is thus omitted.
\end{proof}

\begin{thm} \label{fjcusp}
Let $k$ be half-integral, $N \ge 1$. Let $F \in M_k^n(\Gamma_0^n(N))$ be such that 
its Fourier-Jacobi coefficients $\phi_{F, m}$ as in \eqref{fj-exp} are cusp 
forms for all but finitely many $m \ge 0$.
Then $F$ is a cusp form if $k\geq \frac{n-1}{2}$. 
\end{thm}

\begin{proof}
When $k$ is integral and $\Gamma= \Gamma_n$, this is \cite[Theorem~5.7]{bo-das1}. Let $\mc S$ be the set of indices $m$ for which $\phi_{F,m}$ is not cuspidal. Here we have `$q=1$' in the notation of loc. cit. -- so that $\Lambda_1= \mf N \cup \{ 0 \}$.
Then, from the proof of loc. cit., an application of the Siegel's (opposite) $\widetilde{\Phi}$-operator, one gets that $\displaystyle f= \widetilde{\Phi}(F) = \sumn_{m \in \mc S} \phi_{F,m}(\tau,z) e(m \tau')=0$ as $ k \geq \frac{n-1}{2}$, by exactly the same argument presented in loc. cit. -- viz. by an application of \cite[Lemma~5.6]{bo-das1}.
There is no distinction in the argument between integral and half-integral weights.

Therefore we get that $F(Z)=\sum_{m \ge 1} \phi_{F,m} e(m \tau')$, since $\phi_{F,0}=0$ necessarily by the argument around \eqref{phi00}. Let $g \in P_{n,n-1}(\Q)$ and $R$ as above. We look at the coset representative $gR$ as in \lemref{pnn-1}.
The rest of the proof follows from \lemref{pnn-1} since $F|gR=(F|g)|R $, and if we note that
\begin{align}
F|g = F|   (g_1 \times 1_2) \mid \psmb A & B \\ 0 & D \psme = \left( \sumn_{m \ge 1} (\phi_{F,m}|g_1 )(\tau,z) e(m \tau') \right) \mid \psmb A & B \\ 0 & D \psme = F_1 \mid \psmb A & B \\ 0 & D \psme ,
\end{align}
where $\psmb A & B \\ 0 & D \psme \in \mrm{Sp}_n(\Q)$ and $g_1 \in \mrm{Sp}_{n-1}(\Q)$ -- see eg. \cite[p.~476]{dulinski} for the first equality. The second equality is standard, see e.g. \cite[p.~149]{Ibu-SK}. Now for each $m$,  $ \phi_{F,m}|g_1$ is a cusp form, so $F_1$ has Fourier-expansion supported on positive-definite matrices; and hence so has $F|g$. Clearly then $(F|g)|R$ has the same property. This finishes the proof of the Theorem.
\end{proof}

\section{Further remarks} \label{disc}
\subsection{Question~\ref{main-conj} on `thin sets'} \label{thin}
In this subsection we want to discuss \conjref{main-conj} on `thin sets' -- which are in some sense ``minimal'' subsets $S \subset \glnz \backslash \Lambda^+_n $ on which the bound on the Fourier coefficients will guarantee cuspidality. Observe that in almost all previously known results, say in level $1$, the assumption $S= \glnz \backslash \Lambda^+_n  $ is present -- and crucially used in various ways -- e.g. to invoke Lipschitz's formula or application of the Siegel's $\Phi$ operator. When $S$ is significantly ``thin'', the problem clearly becomes very hard. 

-- \textit{Example}: As an example, in the extreme case that $S$ is just allowed to be infinite, we get counter-examples, as follows. By a result of Kitaoka \cite[Theorem 1.4.5, p. 73]{kitaoka-tata}, if $G \in M^n_k(\Gamma)$ be such that its constant terms are all zero at all the zero dimensional cusps of $\Gamma$, then say for scalar $T=mI_n$ ($m \to \infty$) one actually has the bound
\begin{align}
    a_G(T) \ll m^{n(k-(n+1)/2) - k +n+1} = \det(T)^{k - (n+1)/2 - k/n + (n+1)/n}. 
\end{align}
The above bound is even better than what is needed for \conjref{main-conj} provided $k$ is much bigger than $n+1$, yet $G$ is not cuspidal.

\begin{conj} \label{thin-conj}
    A positive answer to \conjref{main-conj} holds, provided the growth condition is given on a set $S \subset \mcr U(\Gamma) \backslash \Lambda^+_n $ with positive upper density -- i.e., 
    \[  \delta_+(S):=\limsup_{X \to \infty} \frac{ \# \{  T \in S \mid \det(T) \le X \} }{ \# \{T \in \mcr U(\Gamma) \backslash \Lambda_n^+ \mid \det(T) \le X \} } >0. \]
\end{conj}

\subsection{Information from poles} \label{poles}
We prove some results relating the cuspidality of a modular form with the analytic behaviour of its Rankin-Selberg $L$-series.  For a cusp $\omega_j$ of $\Gamma$, put 
\[ R_j(f,s):= R(f|_{\omega_j},s) = \sumn_{n \ge 1} \frac{|a_j(n)|^2}{n^s} ,\]
where $a_j(n)$ are the Fourier coefficients of $f|_{\omega_j}$. Similarly, for $F \in \mkngam$, for $\re(s)$ large enough, define the Rankin-Selberg $L$-series of $F$ by
\begin{equation} \label{rffs}
    \mc R(F,s):= \sumn_{T \in \glnz \backslash \Lambda_n^+}\frac{|a_F(T)|^2}{\epsilon(T) \, \det(T)^s},
\end{equation}
where $\epsilon(T)$ denotes the number of units of $T$. 

\subsubsection{The case \texorpdfstring{$n=1$}{a}}
For the sake of better exposition, and also for a different proof, we first briefly consider the case $n=1$.

\begin{prop}
Let $R_j(f,s)$ denote the Rankin-Selberg $L$-series for $f$ at the cusp $\omega_j$ of $\Gamma$. Then $f \in M_k(\Gamma)$ is a cusp form if and only if none of the $R_j(f,s)$ has a pole at $2k-1$.
\end{prop}

For instance for the level $1$ Eisenstein series $E_k$, it is well known that (see e.g. \cite[p. 821]{bo-ch} or \cite{zag})
\begin{align}
    R(E_k,s) = \frac{\zeta(s) \zeta(s-k+1)^2 \zeta(s-2k+2)}{\zeta(2s-2k+2)},
\end{align}
from which the pole at $s=2k-1$ is visible.

\begin{proof}
    For the proof one considers the $\Gamma$-invariant function $g(z) = y^{k}|f(z)|^2$ and applies the theory of Rankin-Selberg series for functions which are not of rapid decay via renormalization techniques. In particular from the approach of truncation of domain (see \cite{shamita}) it is clear that for a non-cusp form one will have at some cusp
    \begin{align}
        g(z) = \psi(y) + O(y^{-A}) \q \q \text{for all } A>0 \text{ as } y \to \infty,
    \end{align}
with $\psi(y) = |a_f(0)|^2 y^k$ and $a_f(0) \neq 0$; and for a cusp form $\psi(y)=0$ at all cusps. The result now follows directly from the main result in \cite[p.~117]{shamita} -- in fact the residue at $s=2k-1$ is proportional to $|a_f(0)|^2$. Note that our $R(f,s)$ differs from that of \cite{shamita} by the transformation $s \to s-k+1$.
\end{proof}

The following corollary is immediate from \cite{chandra}, if we notice that $2k-1$ is the rightmost pole of $R_j(  f;s)$, where $\omega_j$ is a cusp of $\Gamma$ where $ f$ has a non-zero constant term. 
\begin{cor}
    Let $ f \in M_k(\Gamma)$ be non-zero and that $k \ge 3/2$. Then, $ f$ is not a cusp form if and only if at some cusp $\omega_j$ (see above), one has
\begin{equation} \label{asymp}
\sumn_{n \leq x} |a_{j}(n)|^2 = c_{\mk f, \Gamma} x^{2k-1} +O_{\mk f, \Gamma}(x^\eta) 
\end{equation}
for some constant $c_{\mk f, \Gamma}>0$ not depending on $x$, and for some  $\eta<2k-1$.
\end{cor}

\subsubsection{Higher degrees}
It is natural to wonder what happens in the higher degrees. Let $F \in \mkngam $.
Following \cite{bo-ch}, for $s \in \C$ we put (for a certain `growth-killing' differential operator $\mc D_n$)
\begin{align} \label{rf-diffop}
\mc R^*(F,s) := \begin{cases}
\lan E^n(s +\frac{n+1}{2}-k, \cdot), \, \det(Y)^{k}|F|^2 \ran \q &\text{ if } F \text{ is a cusp form}, \\
\lan E^n(s +\frac{n+1}{2}-k, \cdot), \, \mc D_n(\det(Y)^{k}|F|^2)  \ran \q &\text{ if } F \text{ is a not cusp form}. 
\end{cases}
\end{align}
where 
\[E^n(s,Z) = \sumn_{\Gamma \subseteq P_{n,0}(\z) \backslash \Gamma^n} \im(\gamma(Z))^{-s} \q \text{(with } P_{n,0}(\z) \text{ as in \eqref{pnr} }) 
\]
is the real-analytic Siegel Eisenstein series of degree $n$ with parameter $s$, initially defined and is holomorphic in the region $\re(s)>\frac{n+1}{2}$. Since $\mc D_n(\det(Y)^{k}|F|^2)$ has exponential decay in $Y$, from \eqref{rf-diffop} it is clear that the possible poles of $\mc R^*(F,s)$ can only come from those of the Eisenstein series $E^n(s,\cdot)$.

On the other hand, from e.g., \cite[Remark~3.3]{bo-ch}, \cite[Prop.~2.1]{kalinin}, one knows that 
\begin{align} \label{rf-int-rep}
 \mc R^*(F,s) =\begin{cases}
 c_n(s) \Gamma_n(s) \mc R(F,s)
 \q &\text{ if } F \text{ is a cusp form}, \\
  c_n(s) \Gamma_n(s) \phi_n(s) \phi_n(s+n-2k)  \mc R(F,s) \q &\text{ if } F \text{ is a not cusp form}. 
\end{cases}
\end{align}
where $\Gamma_n(s)$ is the generalized Gamma function defined as $\displaystyle \Gamma_n(s)=\prod\nolimits_{j=0}^{n-1} \Gamma(s-j/2)$, and $ \phi_n(T)$ is a polynomial defined by $\displaystyle \phi_n(T) = \prod\nolimits_{j=0}^{n-1} (T-j/2)$. Here $\displaystyle c_n(s)=2^{1-2ns} \pi^{\frac{n^2-n}{4}-ns}$ is an non-zero entire function of $s$, which is not important for our purposes. We put $\displaystyle \xi(s)=\pi^{-s/2} \Gamma(s) \zeta(s)$. Let us define the completed Eisenstein series by
\begin{align}
    \mathscr E^n(s,Z):= \xi(2s) \prod\nolimits_{j=1}^{\lfloor \frac{n}{2} \rfloor} \xi(4s-2j)E^n(s,Z).
\end{align}
Then it is known that the completed Rankin-Selberg $L$-series of $F \in \mkngam$ defined by
\begin{align} \label{scrF}
    \mathscr R(F,s):=  c_n(s)\Gamma_n(s) \xi(2s +n+1-2k) \prod\nolimits_{j=1}^{\lfloor \frac{n}{2} \rfloor} \xi(4s+2n+2-2k-2j )  \mc R(F,s),
\end{align}
is meromorphic in $\C$, with possible poles at the points $s=k-j/4, \,   0 \le j \le 2n+2$, \textsl{ if} $F$ \textsl{is a cusp form}. Moreover $\mcr R$ satisfies the functional equation given by (see e.g. \cite{kalinin}, \cite{bo-ch})
\begin{align}
    \mathscr R(F,2k- \frac{n+1}{2} -s) = \mathscr R(F,s) \q \text{ if } \Gamma=\Gamma_n, \text{ for all } F \in M^n_k.
\end{align}

Next, we make the following observations about the poles of the function $ \mc R(F,s)$. For a meromorphic function $\mcr G(s)$, let us define
\begin{equation}
\mathscr{P}(\mcr G(s))  = \{ \text{All possible poles of } \mcr G(s) \}.
    \end{equation}
Further, for ease of exposition, define the sets:
\begin{alignat}{2}
    \mathscr{A} &= \mathscr{A}_{n,k} & &:= \{ t_j=k-\frac{j}{4},  \mid 0 \le j \le 2n+2\},\\
    \mathscr{B} &=\mathscr{B}_{n,k} & &:= \{ \frac{\rho}{2} + k - \frac{n+1}{2 },  \ \frac{\rho}{4} + k - \frac{n+1-j}{2} ,  \mid 1 \le j \le \lfloor \frac{n}{2} \rfloor , \, \xi(\rho)=0 \},\\
    \mathscr{C} &=\mathscr{C}_{n,k} & &:=\{ s_j= 2k - \frac{n+j}{2}, \mid 1 \le j \le n)\}.
\end{alignat}

\begin{lem} \label{pole-eisen}
For the weight $0$ (completed) Eisenstein series $\mathscr E^n(s,Z) $, we have,
\[ \mathscr{P}( \mathscr E^n(s,Z) ) \subseteq \{  j/4, \, 0 \le j \le 2n+2 \} . \]
Moreover its right-most pole  is at $s=\frac{n+1}{2}$. It is simple with residue $\prod_{j=1}^{[n/2]} \xi(2j+1)$.
\end{lem}

\begin{proof}
    All of the assertions are known from \cite[Theorem~2]{kalinin-eisen}). 
\end{proof}

\begin{lem} \label{polu}
With the above notation, let $F \in \mkngam$ and $k>n$. Then,
\begin{align}
    \mathscr{P}(\mc R(F,s)) \subseteq \begin{cases}
       \mathscr{A}  \q & F \text{ is a cusp form} \\
       \mathscr{A}  \cup \mathscr{C} \q & F \text{ is not a cusp form}.
    \end{cases}
\end{align}
\end{lem}

\begin{proof}
    When $F$ is a cusp form, the Lemma follows from \eqref{rf-diffop}, \eqref{rf-int-rep} and the knowledge of poles of $\mc R^*(F,s)$, viz. at $t_j=k-j/4, \,   0 \le j \le 2n+2$ as mentioned above. More precisely, if we consider \eqref{rf-diffop}, we see that $\mathscr{P}(\mc R^*(F,s)) = \mathscr{P}( \mathscr E^n(s +\frac{n+1}{2}-k,Z) )$. 
    
    Then from \eqref{rf-int-rep}, it follows that 
    \[ \mathscr{P}(\mc R(F,s)) = \mathscr{P}( \Gamma_n(s)^{-1} \mc R^*(F,s)) =\mathscr{P}(\mc R^*(F,s)) . \]
    To see the second equality above, note that since $k>n$, the zeros of $\Gamma_n(s)^{-1}$ do not cancel any of these poles of $\mc R^*(F,s)$. We already know all the possible poles of $\mc R^*(F,s)$ from \eqref{rf-diffop}, as noted in the above paragraph. Therefore, when $F$ is a cusp form, we have,
    \[  \mathscr{P}(\mc R(F,s)) =  \mathscr{P}( \mathscr E^n(s +\frac{n+1}{2}-k,Z) ) \subseteq \mcr A,\]
if we consult Lemma~\ref{pole-eisen}.

    If $F$ is not a cusp form, we have to a bit more careful, but the same idea works. In this case, we repeat the above process, i.e. look at \eqref{rf-diffop} and \eqref{rf-int-rep}. The extra possible poles come from the polynomial $\phi_n(s+n-2k)$ and these constitute the set $\mcr C$. Note that the zeros of $\Gamma_n(s)^{-1} \phi_n(s)^{-1}$ do not cancel any of these possible poles of $\mc R^*(F,s)$.
\end{proof}

\begin{rmk}
    The reader may note that if $F$ is a cusp form, then $\mc R(F,s)$ has zeros at $s \in \mcr B$, but we will not use this. This is because the non-trivial zeros $\rho$ of $\zeta(s)$ are not rational. 
\end{rmk}

From Lemma~\ref{polu}, one is led to wonder if the ``additional'' possible poles from the set $\mcr C$ might control the cuspidality of $F$. The following proposition says that the above suggestion is correct.

\begin{prop} \label{propole}
    Let $k>n+1$. Then $F \in \mkngam$ is a cusp form if and only if  at some ($0$-dimensional) cusp $\omega$ of $\Gamma$, the Rankin-Selberg $L$-series $\mc R(F,s)$ does not have a pole at any of the points $s_j \in \mcr C$. The same statement holds with $\mc R$ replaced with $\mcr R$.
\end{prop}
 It would be interesting to know whether for each point $s_j \in \mcr C$, there exists an Eisenstein series (or some other modular form) whose Rankin-Selberg series has a pole at $s_j$.

\begin{proof}
One way is clear, as for a  cusp form $F$ one can check that $\mcr R(F,s)$ has no poles at any of the $s_j$ i.e. $s_j \not \in \mcr A$ for all $j$. This follows from the integral representation of $\mc R(F,s)$, as can be read from \cite[Theorem 2]{kalinin}. 

Conversely, let $F \in \mkngam$ be such that $\mc R(F,s)$ has no poles at any of the $s_j$ at some cusp, which we may assume to be $\infty$. By \lemref{polu} then $\mathscr{P}(\mc R(F,s)) \subseteq \mcr C$. The largest of these is $s=k$. Thus by Landau's theorem on Dirichlet series with non-negative Dirichlet coefficients, the Dirichlet series \eqref{rffs} defining $\mc R(F,s)$ would now \textsl{converge} in the region $\re(s)>k$. This is because the zeta factors $\xi(\cdots)$ in \eqref{scrF} are holomorphic and non-vanishing in this region.
By positivity, then we would definitely be able to conclude that $\displaystyle a_F(T) \ll_{F,\epsilon} \det(T)^{k/2 +\epsilon}$ --  for all $T \in \Lambda^+_n$ and $\epsilon>0$. An application of \thmref{mainthm} would then show that  
 $F$ must be a cusp form provided $\displaystyle \frac{k}{2}+\epsilon < k - \frac{n+1}{2}$ or that $k>n+1$.

We argue similarly for $\mcr R(F,s)$. Let $F \in \mkngam$ as above. From \eqref{scrF}, it is evident that the possible poles of $\mcr  R(F,s)$ can be at most at the points in $\mcr A \cup \mcr B$.  One can check as before that none of the $s_j$ can occur in $\mcr A \cup \mcr B$. More precisely, if $s_j \in  \mcr B$ (the other case is considered above), then one gets the relations $k = \rho + \frac{j-1}{2}$ or that $4k +3 = \rho +3(j+j')$ for some $j,j' \in \z$ -- both of these are not possible for any rational $k$ and $\rho$ can't be rational. The rest of the proof is verbatim as that for $\mc R$ and we do not repeat it here. We need only notice that the right-most pole in the set $\mathscr{A}  \cup \mathscr{B} $ is $k$.
\end{proof}

\subsection{Some variants}\label{ext-fj}
In analogy with the usual Fourier expansion, one can ask the following question based on the (extended) Petersson pairing of the Fourier-Jacobi coefficients.

\begin{qs} \label{ext-fj-qs}
    Let $F \in M^n_k$ be such that its Fourier Jacobi coefficients $\phi_{F, \mc M}$ (with $\mc M \in \Lambda_r^+$, $0 \le r \le n-1$) satisfy
    \begin{align}
        |\lan \phi_{F, \mc M},\phi_{F, \mc M} \ran_{reg}| \ll_F \det(\mc M)^{\delta} 
    \end{align}
    for some $\delta<k-\frac{n-r+1}{2}$. Then $F$ is a cusp form.
\end{qs}
Here $\lan \, , \, \ran_{reg}$ denotes the regularized Petersson pairing on $J_{k, \mc M}$, which may not be positive definite. See \cite{bo-das5} for the case $n=2,r=1$ -- here the regularization is known and so is the bound $\displaystyle \lan \phi_{F, m },\phi_{F, m} \ran_{reg} \ll_F m^{k}$ -- i.e., the Hecke bound holds! Whereas for a cusp form $G$ of degree $2$, one expects that $\displaystyle \lan \phi_{G,m},\phi_{G,m} \ran \ll_F m^{k-1}$ for all $m \ge 1$, cf. \cite{ko-fj}.

\subsection{Vector-valued Siegel modular forms} \label{vvsmf}
As we have remarked in the introduction, the lack of a good Fourier-Jacobi expansion makes our approach not immediately useful. We however sketch another possible approach based on restriction of domains -- hopefully it can be made to work in the future.

Let $F \in S_\rho(\Gamma_n)$ where $\rho$ is a polynomial representation of $\glnc$. Following Freitag \cite{freitag1}, we put
$h(z):= \rho(A) \cdot F(A^t z^* A)$, where $z=(z_1,z_2,\ldots,z_n)$, $z^*= \mrm{diag}(z_1,z_2,\ldots,z_n)$, and $A \in \GL_{n}(\mf R)$ chosen in such a way that $\Gamma_A$ ($\subset \sltwor^n$, a certain subgroup, see \cite{freitag1} for the precise definition) is commensurable with a Hilbert modular group over some totally real field $K$ and moreover that none of the components of $h(z)$ is constant. Then it is known from \cite{freitag1} that all of the components of $h$ is a Hilbert modular form with respect to some weights of $\rho$.

The growth condition (say Hecke's bound) for $\rho$ reads as $\norm{\rho(T^{-1/2})a_F(T)} \ll 1$. This should imply a similar growth condition on the Fourier coefficients of $h$, which should imply that each component of $h$, and hence $h$ is cuspidal. Here only parallel weights play a role. Since the set of $A$ described above is dense in $\GL_{n}(\mf R)$, one should be able to prove that $\Phi (F)=0$. The main point is the transition of the growth condition.

\subsection*{Data availability}
Not applicable, as no data were created or analyzed during the study.

\printbibliography

\end{document}